\RequirePackage[l2tabu,orthodox]{nag} 
\documentclass[11pt,headings=standardclasses]{scrartcl}
\usepackage[utf8]{inputenc}

\usepackage[margin=30mm,right=35mm]{geometry}

\usepackage[T1]{fontenc}        
\usepackage{lmodern}            
\usepackage[intlimits]{amsmath} 

\usepackage{hyperref}           
\usepackage{grffile}            

\usepackage[ugly]{units}        

\usepackage{url}                
\usepackage{breakurl}           
\usepackage{xspace}             
\usepackage{xcolor}             
\usepackage{booktabs}           

\usepackage{listings}           

\newtheorem{theorem}{Theorem}
\newtheorem{remark}[theorem]{Remark}
\newtheorem{corollary}[theorem]{Corollary}

\newenvironment{proof}{\noindent{\bfseries Proof.\/}}{}
\usepackage{graphicx}
\usepackage{ifthen}
\usepackage{amsmath}
\usepackage{amssymb}
\usepackage{latexsym}

\usepackage{bbm}
\usepackage{color}
\usepackage{enumitem}

\newcommand{\real}{\mathbbm{R}}
\newcommand{\naturalnumbers}{\mathbbm{N}}

\renewcommand{\hat}{\widehat}
\renewcommand{\tilde}{\widetilde}
\renewcommand{\bar}{\overline}

\newcommand{\yfast}[1][undefined]{y_{F\ifthenelse{\equal{#1}{undefined}}{}{,\,#1}}}

\def\dyfast{\dot{y}_{F\hspace*{-0.1mm}}}
\def\tyfast{\tilde{y}_{F\hspace*{-0.1mm}}}
\def\hyfast{\hat{y}_{F\hspace*{-0.1mm}}}

\newcommand{\yslow}[1][undefined]{y_{S\ifthenelse{\equal{#1}{undefined}}{}{,\,#1}}}
\def\dyslow{\dot{y}_S}
\def\tyslow{\tilde{y}_S}
\def\hyslow{\hat{y}_S}

\def\ffast{f_F}
\def\fslow{f_{S\hspace*{0.2mm}}}

\def\myend{\text{\scriptsize end}}

\def\eps{\ifmmode{\varepsilon} \else{\varepsilon}\fi }

\def\one{{\ifmmode{{\text{1}}\mkern-4.0mu{\rm l}\mkern1.0mu\mbox{}}
          \else\leavevmode\hbox{$\text{\rm 1}\mkern-4.0mu{\rm                 
           l}\mkern4.0mu\mbox{}$}\fi}}

\begin{document}

\title{Inter/extrapolation-based multirate schemes --- a dynamic-iteration perspective\protect\thanks{The authors are indebted to the EU project ROMSOC (EID). 
}}

\author{Andreas Bartel and  Michael G\"unther \\[1ex]
            \small  University of Wuppertal\\[-0.75ex]
            \small  Faculty of Mathematics and natural sciences \\[-0.75ex]
           \small   IMACM\\[-0.75ex]
	\small40097 Wuppertal (Germany)  \\[-0.75ex]
           \small  {\{bartel,guenther\}@uni-wuppertal.de}           %
}


\maketitle

{\bfseries Abstract. }
Multirate behavior of ordinary differential equations (ODEs) and different\-ial-alge\-braic equations (DAEs) is characterized by 
widely separated time constants in different components of the solution or different additive terms of the right-hand side.
Here, classical multirate schemes are dedicated solvers, which apply (e.g.) micro and  macro steps to resolve fast and slow changes in a transient simulation accordingly.
The use of extrapolation and interpolation procedures is a genuine way for coupling the different parts, which are  defined on different time grids. \\
This paper contains for the first time, to the best knowledge of the authors, a complete convergence theory for inter/extrapolation-based multirate schemes for both ODEs and DAEs of index one, which are based on the fully-decoupled approach, the slowest-first and the fastest-first approach. 
The convergence theory is based on linking these schemes to multirate dynamic iteration schemes, i.e., dynamic iteration schemes without further iterations. This link defines  naturally stability conditions for the DAE case.


\section{Introduction}
\label{intro}
In practice, technical applications are often modeled as coupled systems of ordinary differential equations (ODEs) or differential algebraic equations (DAEs). Furthermore,  it is a very common aspect of technical applications that the transient behavior is characterized by different time constants. At a given instance of time, certain parts of a dynamical system are  slowly evolving, while others have a fast dynamics in the direct comparison. Here, this is referred to 
\emph{multirate behavior}. To name but a few applications:  multibody systems~\cite{Arnold_2007,Eich1998}, electric circuits~\cite{El_Guennouni_2007,GuentherFeldmannMaten2005}, climate models~\cite{Stocker2011} 
and, of course, multiphysical systems, e.g. field/circuit coupling~\cite{Schoeps_2010a}.
\color{black}
Now, to have an efficient numerical treatment of systems with multirate behavior, special integration schemes are developed, so-called multirate schemes. 
To the best knowledge of the authors, the multirate history goes back to Rice~\cite{Rice60} in 1960, where step sizes for time integration are adapted to the activity level of subsystems.
Many work followed, and we give only a partial list here: %
based on BDF-methods~\cite{Gear_1984}, 
based on ROW methods~\cite{GuentherRentrop_1992}, 
based on extrapolation methods~\cite{Engstler_1997}
partitioned RK and compound step~\cite{Guenther_2001}, 
mixed multirate with ROW~\cite{Bartel_2001}, 
based on a refinement strategy~\cite{Savcenco_2007}, 
for conservation laws~\cite{Constantinescu2007},
compound-fast~\cite{Verhoeven_2008},
infinitesimal step~\cite{Wensch_2009},
implicit-explicit~\cite{Constantinescu_2010},
based on GARK-methods~\cite{GuentherSandu_2015}.

The fundamental idea of a multirate scheme is the following: an efficient algorithm should (if there are no stability issues) sample a certain component/subsystem according to the activity level. The more active a component is, the shorter are the time scales and the higher the sampling rate should be chosen to achieve a given level of accuracy. In other words, there is not a global time step, but a local one, which should reflect the inherent time scale of an unknown or some subsystem. For simplicity, we work here with only two time scales. That is, we allow for an fast subsystem (of higher dynamics), which employs a small step of size $h$ (\emph{micro step}) and a slow subsystem, which employs a larger step size $H$ (\emph{macro step}). 
Furthermore, we assume for simplicity the relation $H=m h$ with $m \in \naturalnumbers$.
In fact, the main feature of a certain multirate scheme is to define the coupling variables in an appropriate way. Here we focus on inter- and extrapolation strategies for coupling both subsystems, since we aim at highlighting the connection to dynamic iteration schemes.

The work is structured as follows:  In Sec~\ref{sec:notation}, the formulation of multirate initial value problems is given on the basis of ordinary differential equations (ODEs). Furthermore, various known versions of extra- and interpolation coupling is explained. Following this, the consistency of multirate one-step methods are discussed for ODEs (Sec.~\ref{sec:ODE}). Then, in Sec.~\ref{sec:dae}, the ODE results are generalized to the DAE case. Conclusions complete the presentation. 

\section{Notation for coupled systems and multirate extra/interpolation}
\label{sec:notation}
We start from an initial value problem (IVP) based on a model of ordinary differential equations (ODEs):
\begin{equation}
\label{ivp.ode.original}
\dot w = h(t,w), \qquad w(t_0)=w_0, \qquad 
t\in (t_0,\, t_{\myend}],
\end{equation}
where $h$ is continuous and Lipschitz continuous in $w$, $w_0 \in \real^n$ is given. Moreover, let $h$ or $w$, resp., be comprised of some slower changing parts (in time domain), whereas the remaining parts are faster changing. This is referred to as multirate behavior 
Now, there are two equivalent ways of partitioning:
\begin{enumerate}[label=\alph*)]
    \item The \textit{component-wise partitioning} splits the unknown into slow $\yslow(t)\in \real^{m}$ and fast components
     $\yfast(t)\in \real^{n-m}$, such that 
    $w^\top = (\yslow^\top,\yfast^\top)$  and 
    \begin{equation}
        \label{ode.part.comp}
        \begin{array}{rlrl}
            \dyslow & =  \fslow(t, \, \yslow, \, \yfast), \qquad & \yslow(t_0) & = \yslow[0],\\
            \dyfast & =  \ffast(t, \, \yslow, \, \yfast), \qquad &  \yfast(t_0) & = \yfast[0],
        \end{array}
    \end{equation}
   with corresponding splitting of the right-hand side.
   
    \item The \textit{right-hand side partitioning} is an additive splitting of $h$ into slow and fast summands:
        \begin{equation}
            \label{ode.part.rhs}
            \dot w = h_s(t, \, w)+h_f(t, \, w), \qquad w(t_0)=w_0,
        \end{equation}
        such that $w=w_s+w_f$ with $\dot{w}_s = h_s(t,\, w_s + w_f)$ and 
        $\dot{w}_f = h_f(t,\, w_s + w_f)$. Of course, the initial data needs to be split in a suitable way. If the dynamics are solely determined by $h_s$ and $h_f$, the splitting is arbitrary to some extent.
\end{enumerate}

\noindent Since both ways of partitioning are equivalent, we choose for the work at hand the formulation \eqref{ode.part.comp}, without loss of generality. Moreover, the partitioning \eqref{ode.part.comp} can be generalized to the case of differential algebraic equations (DAEs) with certain index-1 assumptions. This DAE setting is treated in Sec.~\ref{sec:dae}.

\color{black}

In this work, we study  multirate methods, which belong to the framework of one-step-methods (and multi-step schemes, too, see remark~\ref{remark.also.mr} below) \color{black} and which are based on extrapolation and interpolation for the coupling variables. 
 To describe these methods,  
 let us assume that the computation of the coupled system~\eqref{ode.part.comp} has reached time $t=\bar{t}$ with 
     \begin{equation}
        \label{eq:ode.part.comp.intermediate}
        \begin{array}{rlrl}
            \dyslow & =  \fslow(t, \, \yslow, \, \yfast), \qquad & \yslow(\bar{t}) & = \yslow[\bar{t}],\\
            \dyfast & =  \ffast(t, \, \yslow, \, \yfast), \qquad &  \yfast(\bar{t}) & = \yfast[\bar{t}].
        \end{array}
    \end{equation}
Now, the multirate integration of the whole coupled system is defined for one macro step, i.e., on $[\bar{t},\,\bar{t}+H] \subseteq [t_0,\,t_{\myend}]$. It comprises a single step of macro step size $H$ for the subsystem $\yslow$ and $m\in\naturalnumbers$ steps of (micro step) size $h$ for $\yfast$.
To this end, the respective coupling variables need to be evaluated. Here, our presentation is restricted to extrapolation and interpolation for the coupling variables, although there are several other techniques.
Depending on the sequence of  computation of the unknowns \color{black} $y_S$ and $y_F$, 
one distinguishes the following three versions  of extra-/and interpolation techniques\color{black}:  
 \begin{enumerate}[label=\roman*)]
     \item \emph{fully-decoupled approach~ \cite{Bartel_Guenther_2019}:} fast and slow variables are integrated in parallel using in both cases \color{black} extrapolated waveforms based on information from the 
     initial data of the current macro step at $\bar{t}$;
     \color{black}

     \item \emph{slowest-first approach~\cite{Gear_1984}:} in a first step, the slow variables are integrated, using an extrapolated waveform of $y_F$ based on information  available at $\bar{t}$ \color{black}
    for evaluating the coupling variable $y_F$ in the current macro step. 
    In a second step, $m$ micro steps are performed to integrate the fast variables $y_F$ from $\bar t$ to $\bar t + H$, using an interpolated waveform of $y_S$ based on information from the current macro step size $[\bar t,\bar t+H]$ for evaluating the coupling variable $y_F$.
    
     \item \emph{fastest-first approach~\cite{Gear_1984}:} in a first step, $m$ micro steps are performed to integrate the fast variables, using an extrapolated waveform of $y_S$ based on information 
     available at $\bar{t}$ \color{black}
     for evaluating the coupling variable $y_S$ in the current macro step. 
     In a second step, one macro step is performed to integrate the slow variables $y_S$ from $\bar t$ to $\bar t + H$, using an interpolated waveform of $y_F$ based on information from the current macro step size $[\bar t,\bar t+H]$ for evaluating the coupling variable $y_F$.
 \end{enumerate} 
 
 \begin{remark}\label{remark.end.ch2}
 The restriction that the extrapolation can only be based on the information at $\bar{t}$ can be relaxed to the data of the preceding macro step $[\bar{t}-H,\, \bar{t}]$. In fact, one can encode such an information e.g. as a spline model, which is also updated and transported from macro step to macro step. 
 \end{remark}
 \color{black}

\section{The ODE case}
\label{sec:ODE}
\nopagebreak%
The details presented in this section are based on a result first presented in~\cite{Bartel_Guenther_2019}. Starting from this result, we use the underlying strategy to extend it to our case of the three multirate versions named in the previous section.
Basically, for ODE systems, all variants of extrapolation/interpolation-based multirate schemes have convergence order $p$  (in the final asymptotic phase) provided that  it holds: \color{black}
\begin{enumerate}[label=\roman*)]
\item the basic integration scheme  (i.e., the scheme for both the slow and the fast subsystems with given coupling data) \color{black} has order $p$ and 

\item
the extrapolation/interpolation schemes are of approximation order $p-1$. 
\end{enumerate}
 For the fully decoupled approach, this is a consequence of the following result, which is a particular case of a more general setting presented in~\cite{Bartel_Guenther_2019}\color{black}:
\begin{theorem}[Consistency  of fully-decoupled multirate schemes]\label{thm:general-multirate}
Given the coupled ODE-IVP \eqref{ode.part.comp}, where $\fslow$ and $\ffast$ are Lipschitz w.r.t. the sought solution. Furthermore, we apply two basic integration schemes of order $p$: one for $\yslow$ with macro step size $H$, a second for $\yfast$ with fixed multirate factor $m(\in\naturalnumbers)$ steps of size $h$. If these integration schemes are combined with two extrapolation procedures for the coupling variables of order $p-1$, the resulting fully decoupled multirate scheme has order $p$. 
\end{theorem}

Since the strategy of the proof is needed for the further new results, we present the proof in details, although a slightly more general version can be found in \cite{Bartel_Guenther_2019}.
\color{black}

\begin{proof}
We consider the case that we have computed the IVP system~\eqref{ode.part.comp} until time $\bar{t}$ with initial data $\yslow(\bar t) = \yslow[{\bar t}]$,  $\yfast(\bar t) = \yfast[{\bar t}]$, i.e., we have the setting given in  system~\eqref{eq:ode.part.comp.intermediate}. Moreover, the unique solution of \eqref{eq:ode.part.comp.intermediate} is referred to as  
\[
(\yslow(t; \, \yslow[{\bar t}], \, \yfast[{\bar t}])^{\!\top}\!\!, \;\,\, 
  \yfast(t; \, \yslow[{\bar t}], \, \yfast[{\bar t}])^{\!\top}) %
  \;\quad \mbox{ or }\;\quad (\yslow(t)^{\!\top}\!\!, \;\,\, \yfast(t)^{\!\top}) \mbox{ as short-hand}.
\]
Next, we provide extrapolated, known quantities $\tyslow$ and $\tyfast$ for the coupling variables of order $p-1$: (for constants respective $L_S,\, L_F >0$)
\begin{equation}\label{eq:extrapolation-approx}
\begin{array}{rll}
\yslow(t)-\tyslow(t) &= L_S \cdot H^p + {\cal O}(H^{p+1}) \;\quad 
&\text{for any } \, t \in [\bar t,\, \bar t+H], \qquad 
\text{ and } \;  \\
    \yfast(t)-\tyfast(t) &= L_F \cdot H^p + {\cal O}(H^{p+1})
\qquad &\text{for any } \, t \in [\bar t,\, \bar t+H].
\end{array}
\end{equation}
Replacing the coupling variables in~\eqref{eq:ode.part.comp.intermediate}
by $\tyslow$ and $\tyfast$, we obtain the following modified system
\begin{equation} \label{mod.ivp.ode}
\begin{array}{rlrl}
\dyslow & = \fslow(t, \, \yslow, \, \tyfast)=: \tilde{f}_S (t, \, \yslow), 
             \qquad & \yslow(\bar t) = \yslow[{\bar t}],
             \\
\dyfast & =  \ffast(t, \, \tyslow, \, \yfast)=: \tilde{f}_F (t, \, \yfast), 
             \qquad & \yfast(\bar t) = \yfast[{\bar t}], 
             \end{array}
\end{equation}
which is fully decoupled (for $t\in [\bar{t},\, \bar{t} +H]$). Its unique solution is referred to as
\[ (\hyslow(t; \,\, \yslow[\bar t], \, \yfast[\bar t])^{\!\top}\!\!, \;\,\, 
      \hyfast(t; \,\, \yslow[\bar t], \, \yfast[\bar t])^{\!\top}).
\]     

Now, we apply the two basic integration schemes of order $p$ in multirate fashion to the decoupled model~\eqref{mod.ivp.ode} and we refer to the numerical solution at $t^{\ast}=\bar{t} + H$ as 
\[
(\yslow[H](t^*),\, \yfast[H](t^*))^{\!\top} .
\]
Then, the distance between multirate and exact solution can be estimated as follows:        
\begin{align}\label{eq:split-error-general}
\begin{pmatrix} \| \yslow[H] (t^*) - \yslow (t^*) \| \\ 
                \| \yfast[H] (t^*) - \yfast (t^*) \|
\end{pmatrix} 
& \leq 
 \begin{pmatrix} \| \yslow[H] (t^*) - \hyslow (t^*) \| \\ 
                 \| \yfast[H] (t^*) - \hyfast (t^*) \|
\end{pmatrix} 
 +
 \begin{pmatrix} \| \hyslow (t^*) - \yslow (t^*)\| \\ 
                 \| \hyfast (t^*) - \yfast (t^*) \|
\end{pmatrix}.
\end{align}
The fully decoupled multirate scheme gives for the first term on the right-hand side:
\begin{align}\label{eq:integration-error-general}
 \begin{pmatrix} \| \yslow[H](t^*) - \hyslow (t^*) \| \\ 
                 \| \yfast[H](t^*) - \hyfast (t^*) \|
\end{pmatrix} 
 \le   
 \begin{pmatrix}
     c_S H^{p+1} + \mathcal{O} (H^{p+2})  \\ 
     c_F H^{p+1} + \mathcal{O} (H^{p+2})
 \end{pmatrix} 
\end{align}
employing constants $c_S, \, c_F >0$ (for leading errors). 
Using Lipschitz continuity of $\fslow,\, \ffast$ for the second summand on the right-hand side of \eqref{eq:split-error-general}, we find
\begin{align} \nonumber
\begin{pmatrix} \| \hyslow (t^*) - \yslow (t^*) \| \\ 
                \| \hyfast (t^*) - \yfast (t^*) \| 
\end{pmatrix} 
 & \leq  
\int_{\bar t}^{t^*}  \!\!
\begin{pmatrix}
  \| \fslow \bigl(\tau,\, \hyslow (\tau),\, \tyfast (\tau)\bigr) 
   - \fslow \bigl(\tau,\, \yslow (\tau),\,  \yfast (\tau)\bigr) \| \\ 
  \| \ffast \bigl(\tau,\, \tyslow (\tau),\, \hyfast (\tau)\bigr) 
  -  \ffast \bigl(\tau,\, \yslow (\tau),\,  \yfast (\tau)\bigl) \|
\end{pmatrix} 
d \tau \\
& \leq  
\int_{\bar t}^{t^*} \!\!
\begin{pmatrix}
  L_{S,S} \|\hyslow (\tau) - \yslow (\tau) \| + L_{S,F} \|\tyfast (\tau) - \yfast (\tau) \| 
     \\
  L_{F,S} \|\tyslow (\tau) - \yslow (\tau) \| + L_{F,F} \|\hyfast (\tau) - \yfast (\tau) \|  
\end{pmatrix} d \tau 
\label{eq:estimate_hat}
\end{align}
with respective Lipschitz constants $L_{i,j}$ (for system $i$ and dependent variables $j$). We remark that this estimate is decoupled. Inserting the extrapolation estimates \eqref{eq:extrapolation-approx}, we deduce further
\begin{align*}
\begin{pmatrix} \!\| \hyslow (t^*) - \yslow(t^*) \|\! \\[4ex] 
                \!\| \hyfast (t^*) - \yfast(t^*) \| \!
\end{pmatrix} 
& \leq  
\begin{pmatrix} \displaystyle
\! L_{S,F} \cdot L_{F} \cdot 
    H^{p+1} 
   +  L_{S,S} \!\int\limits_{\bar t}^{\mbox{}\,\,\,t^*} \!\! \|\hyslow(\tau) - \yslow(\tau) \| d \tau 
  + {\cal O}(H^{p+2})\! \\[0.5ex] \displaystyle
\! L_{F,S} \cdot L_{S}  \cdot H^{p+1}  
  +   L_{F,F} \!\int\limits_{\bar t}^{\mbox{}\,\,\,t^*} \!\! \|\hyfast(\tau) - \yfast(\tau) \| d \tau 
    + {\cal O}(H^{p+2})   \!
\end{pmatrix} \!.
\end{align*}
Via Gronwall's lemma, we deduce:
\begin{align}
  \begin{pmatrix} \| \hyslow (t^*) - \yslow (t^*) \| \\[0.25ex]
                  \| \hyfast (t^*) - \yfast (t^*) \| 
\end{pmatrix} 
& \leq  
\begin{pmatrix} 
 L_{S,F}  L_{F} 
 \,  e^{L_{S,S}  (t^* -\bar{t})} \,  H^{p+1}  + {\cal O}(H^{p+2}) \\[0.25ex]
     L_{F,S}  L_{S} \, e^{L_{F,F}  (t^* -\bar{t})}  \, H^{p+1}    + {\cal O}(H^{p+2})
\end{pmatrix} .
\label{eq:final_fully_decoupled}
\end{align}
In combination with the integration estimate~\eqref{eq:integration-error-general},  the error \eqref{eq:split-error-general} of the fully-decoupled multirate scheme has consistency order $p$ on the macro scale level, which is the claim.
\hfill $\Box$
\end{proof}
\color{black}

The proof can be slightly adapted to verify the convergence result for the both remaining variants as well:
\begin{corollary}[Consistency of slowest-first multirate schemes]
\label{thm:general-multirate2}
The convergence result of Theorem~\ref{thm:general-multirate} 
 remains 
valid if the fully-decoupled approach is replaced by the slowest-first approach, i.e., the coupling variables $\yslow$  (during the integration of  $\yfast$) 
are evaluated using interpolation  of the already computed slow data in the current macro step.  %
\end{corollary}

\begin{proof} We just give the changes of the above proof. 
For the slowest-first variant, the modified equation on the current macro step $[\bar{t},\, \bar{t}+H]$
reads
\begin{equation} \label{mod.ivp.ode.2}
\begin{array}{rlrl}
\dyslow & = \fslow(t, \, \yslow, \, \tyfast)=: \tilde{f}_S (t, \, \yslow), 
             \qquad & \yslow(\bar t) = \yslow[{\bar t}],
             \\
\dyfast & =  \ffast(t, \, {\color{black} \yslow^{int}}, \, \yfast)=: \tilde{f}_F ({\color{black} t, \, \yfast}), 
             \qquad & \yfast(\bar t) = \yfast[{\bar t}]
             \end{array}
\end{equation}
\color{black}
with extrapolated values $\tyfast$ as in the fully-decoupled approach and interpolated values $\yslow^{int}$ of order $p-1$ based on the numerical approximations $y_{S,H} 
\color{black} (t_k)$ \color{black} with $t_k \in [\bar t, \bar{t} +H]$ such that it holds: \color{black}
\begin{equation}\label{eq:interpolated-values-slow}
\hyslow(t)-\yslow^{int}(t) = \widetilde L_S \cdot H^p + {\cal O}(H^{p+1}) \;\quad 
\text{for any } \, t \in [\bar t,\, \bar t+H].
\end{equation}
\color{black}


\noindent Again, the hat-notation is again employed for the exact solution of system \eqref{mod.ivp.ode.2}. 
The computation of the slow part still employs extrapolated coupling variables. 
This decouples the slow part from the fast part as before 
and hence the error estimates of $\yslow$ are unchanged. 
In fact, we can use the 
estimates (\ref{eq:integration-error-general}${}_1$) and (\ref{eq:final_fully_decoupled}${}_1$): 
for any time $\tau \in (\bar{t},\, \bar{t}+H]$.

Now, for the fast part, the corresponding estimate to (\ref{eq:estimate_hat}${}_2$) reads
\color{black}
(with using $\yslow^{int}(t)-\yslow(t)= \yslow^{int}(t)-\hyslow(t)+\hyslow(t)-\yslow(t)$
\color{black}
\begin{eqnarray*}
                \| \hyfast (t^*) - \yfast (t^*) \| 
 & \leq  &
\int_{\bar t}^{t^*} \!\!
  L_{F,S} \Bigl( {\color{black} \|\hyslow (\tau) - \yslow^{int} (\tau) \|} + \|\hyslow (\tau) - \yslow (\tau) \|   \color{black}\Bigr)\color{black} \\
  & &  \quad + \, L_{F,F} \|\hyfast (\tau) - \yfast (\tau) \|  
 \, d \tau .
\end{eqnarray*}
Using~(\ref{eq:final_fully_decoupled}${}_1$) (with $\tau$ instead of $t^\star)$ and \color{black} using \eqref{eq:interpolated-values-slow}\color{black}, we find 
\begin{align*}
                &\| \hyfast (t^*) - \yfast (t^*) \| 
                \\
 &\quad  \leq  
\int_{\bar t}^{t^*} \!\!
  \left( {\color{black} L_{F,S} \widetilde L_S H^{p} + \mathcal{O}(H^{p+1})} + 
  L_{F,S} L_{S,F}  L_{F}    \,  e^{L_{S,S}  (\tau -\bar{t})} \,  H^{p+1}  + \mathcal{O}(H^{p+2}) \right. \\
  & \qquad\qquad  +\, L_{F,F} \|\hyfast (\tau) - \yfast (\tau) \|  \Bigr) \,  d \tau \\
 & 
\quad \leq  {\color{black} L_{F,S} \widetilde L_S H^{p+1}} +
  L_{F,F} \!\int_{\bar t}^{t^*} \!\!
   \|\hyfast (\tau) - \yfast (\tau) \|  
 d \tau  + \mathcal{O}(H^{p+2}) .
\end{align*}
\color{black}

Now, the application of Gronwall's lemma leads to
$$
\| \hyfast (t^*) - \yfast (t^*) \| 
 \leq {\color{black} L_{F,S} \widetilde L_S }  e^{L_{F,F}H}H^{p+1}
 +{\color{black} \mathcal{O}(H^{p+2})}.
$$
 Finally, we need to form the total error in the fast components, the equivalent to (\ref{eq:split-error-general}${}_2$). Since the numerical scheme for the fast component is of order $p$, we can still employ (\ref{eq:integration-error-general}${}_2$), and
we get the estimate
\begin{eqnarray}
    \label{slowest.first.estimate2}
  \| \yfast[H] (t^*) - \yfast(t^*) \| 
& \leq & \left(
        c_F {\color{black} + L_{F,S} \widetilde L_S } \color{black}e^{{L_{F,F}H}}\color{black}
    \right) H^{p+1} + {\cal O}(H^{p+2}).
\end{eqnarray}

\hfill $\Box$

\begin{remark}
If one uses interpolation schemes of order $p$ instead of $p-1$, which is the case if dense output is used within embedded Runge-Kutte schemes, for example, one has to replace the term $\widetilde L_S H^p$ by
$\widetilde L_S H^{p+1}$, which yields the estimate
\begin{eqnarray}
  \label{slowest.first.estimate2.2}
\| \yfast[H] (t^*) - \yfast(t^*) \|  
& \le &  c_F  H^{p+1} + {\cal O}(H^{p+2}),
\end{eqnarray}
that is, the extra-/interpolation error is dominated by the error of the numerical integration scheme. 
\end{remark}

\hfill $\Box$
\end{proof}

\begin{corollary}[Consistency of fastest-first multirate schemes]
\label{thm:general-multirate3}
The convergence result of Theorem \ref{thm:general-multirate} remains valid \color{black} if the fully-decoupled approach is replaced by the fastest-first one, i.e., the coupling variables $\yfast$ (during the integration of $y_S$) \color{black} 
are evaluated using interpolation instead of extrapolation.
\end{corollary}

\begin{proof}
For the fastest-first variant, the modified equation~\eqref{mod.ivp.ode} reads
 on $[\bar{t},\, \bar{t}+H]$\color{black} 
\begin{equation} \label{mod.ivp.ode.3}
\begin{array}{rlrl}
\dyslow & = \fslow(t, \, \yslow, \, {\color{black} \yfast^{ext}})=: \tilde{f}_S ({\color{black} t, \, \yslow} ), 
             \qquad & \yslow(\bar t) = \yslow[{\bar t}],
             \\
\dyfast & =  \ffast(t, \, \tyslow, \, \yfast)=: \tilde{f}_F (t, \, \, \yfast), 
             \qquad & \yfast(\bar t) = \yfast[{\bar t}],
             \end{array}
\end{equation}
\color{black}
with extrapolated values $\tyslow$ as in the fully-decoupled approach and interpolated values $\yfast^{int}$ of order $p-1$ based on the numerical approximations $y_{F,H}\color{black}(t_k)$ \color{black} with $t_k  \in [\bar t, \bar{t} + H]$:\color{black}
\begin{equation}\label{eq:interpolation-fast}
\hyfast(t)-\yfast^{int}(t) = \widetilde L_F \cdot H^p + {\cal O}(H^{p+1}) \;\quad 
\text{for any } \, t \in [\bar t,\, \bar t+H].
\end{equation}
\color{black}
Here, the second equation for $\yfast$ is unchanged with respect to Theorem~\ref{thm:general-multirate}, since the extrapolation of $\yslow$ is still used. 
Hence, we still have all  respective estimates for the fast part, in particular  (\ref{eq:integration-error-general}${}_2$) and   (\ref{eq:final_fully_decoupled}${}_2$).
For the slow part, the corresponding estimate to (\ref{eq:estimate_hat}${}_1$) now  reads 
\color{black}
(with using $\yfast^{int}(t)-\yfast(t)= \yfast^{int}(t)-\hyfast(t)+\hyfast(t)-\yfast(t)$)
\color{black}
\begin{eqnarray*}
                \| \hyslow (t^*) - \yslow (t^*) \| 
& \leq  &
\int_{\bar t}^{t^*} \!\! \left(L_{S,F} \color{black} \left( \color{black} \|\hyfast (\tau) - \yfast^{int} (\tau) \| \color{black} +
   \|\hyfast (\tau) - \yfast (\tau) \| \color{black} \right) \color{black}
  \right. 
  \\
  && \qquad \left.
 + \,   L_{S,S} \|\hyslow (\tau) - \yslow (\tau) \| \right)  
 d \tau .
\end{eqnarray*}
Using~\ref{eq:final_fully_decoupled}${}_2$ (with $\tau$ replaced by $t^\star$) and using \eqref{eq:interpolation-fast},  we find 
\begin{align*}
                \| \hyslow (t^*) - 
                &
                \yslow (t^*) \| 
                 \leq  
                \int_{\bar t}^{t^*} \!\!
  \Bigl( {\color{black} L_{S,F} \widetilde L_F H^p+  \mathcal{O}(H^{p+1}})+ L_{S,F} L_{F,S}  L_{S}    \,  e^{L_{F,F}  (\tau -\bar{t})} \,  H^{p+1}   \\
  &  
  \qquad\qquad\qquad + \, \mathcal{O}(H^{p+2}) 
  + L_{S,S} \|\hyslow (\tau) - \yslow (\tau) \|  \Bigr)
 d \tau \\
 & \leq   {\color{black} L_{S,F} \widetilde L_F H^{p+1}} +
 \int_{\bar t}^{t^*} \!\!
   L_{S,S} \|\hyfast (\tau) - \yfast (\tau) \|  
 d \tau +  \mathcal{O}(H^{p+2}). 
\end{align*}
Applying now Gronwall's lemma leads to
$$
\| \hyslow (t^*) - \yslow (t^*) \| 
 \leq {\color{black} L_{S,F}\widetilde L_F }   e^{L_{S,S}H}H^{p+1} + {\color{black} \mathcal{O}(H^{p+2})}.
$$
Finally, we use both the above deduced error and the numerical error~(\ref{eq:integration-error-general}${}_1$) in the  general error sum~(\ref{eq:split-error-general}${}_1$) \color{black} and we find  for the slow part 
\begin{eqnarray}
    \label{fastest.first.estimate2}
  \| \yslow[H] (t^*) - \yslow (t^*) \|  
& \leq & \Bigl( c_S + {\color{black} L_{S,F} \widetilde L_F} \color{black} e^{L_{S,S}H} \color{black} \Bigr)H^{p+1} + {\cal O}(H^{p+2}).
\end{eqnarray}

\begin{remark}
If one uses interpolation schemes of order $p$ instead of $p-1$, which is the case if dense output is used within embedded Runge-Kutte schemes, for example, one has to replace the term $\widetilde L_F H^p$ by
$\widetilde L_F H^{p+1}$, which yields the estimate
\begin{eqnarray}
  \label{fastest.first.estimate2.2}
  \| \yslow[H] (t^*) - \yslow (t^*) \|  
& \le & c_S H^{p+1} + {\cal O}(H^{p+2}),
\end{eqnarray}
that is, the extra-/interpolation error is dominated by the error of the numerical integration scheme. 
\end{remark}

\end{proof}

\begin{remark}
\label{remark.also.mr}
 For the basic integration schemes employed in Thm.~\ref{thm:general-multirate}, Cor.~\ref{thm:general-multirate2} and
 Cor.~\ref{thm:general-multirate3}
 we can use either

\begin{enumerate}
 \item[a)]  one-step integration schemes, or
 \item[b)] multistep schemes, where both schemes are $0$-stable.
 \end{enumerate}
\end{remark}
\begin{remark}[Schemes]
Extrapolation of order 0 and 1 can be easily obtained from the initial data at $t=\bar{t}$ and a derivative information, which is provided by the ODE. This allows directly the construction of multirate methods of order 2.
\end{remark}

\begin{remark}
  Notice that for a working multirate scheme, we still have to 
 specify the extrapolation/interpolation  formulas. In fact, 
  arbitrary high orders of the extra-/inter\-po\-la\-tion are only possible 
  if information of previous time steps is used. Generally, this may turn a one-step 
  scheme \color{black} into a multi-step scheme, and raise questions concerning stability. However, if the extrapolation is computed sequentially in a spline-oriented fashion  (see Remark~\ref{remark.end.ch2}), the modified functions $\tilde f_S$ and $\tilde f_F$ are the same for all time intervals inside $[t_0,t_\myend]$, \color{black}
  and the extrapolation/interpolation based multirate scheme can still be considered as a one-step scheme applied to the modified ODE equations.
\end{remark}

\color{black}

\section{The DAE case}
\label{sec:dae}

The component-wise partitioning~\eqref{ode.part.comp} (as well as the right-hand side partitioning~\eqref{ode.part.rhs}) can be generalized to the case of differential algebraic equations (DAEs). Let us assume that the slow and the fast subsystem can be written as semi-explicit system of index-1, each for given corresponding coupling terms as time functions. This reads:
 \begin{alignat}{2} \nonumber
  \dot y_S &=  f_S(t,\, y_S,\, y_F,\, z_S,\, z_F), \; \, y_S (t_0)=y_{S,0},
\;\;  &
  \dot y_F  &=  f_F(t,\, y_S,\, y_F,\, z_S,\, z_F) , \;\,  y_F (t_0)=y_{F,0},
  \\ %
         0 &=  g_S(t,\, y_S,\, y_F,\, z_S,\, z_F), 
         &
\quad
          0 &=  g_F(t,\, y_S,\, y_F,\, z_S,\, z_F).
          \label{eq:DAE-split-formulation}
\end{alignat}
Moreover, the overall system is assumed to be index-1 as well. All index-1 conditions lead to the assumption that the following Jacobians 
\begin{equation}\label{eq:cosim-index-one-conditions}
\frac{\partial g_S}{\partial z_S }, \quad
\frac{\partial g_F}{\partial z_F } \quad \mbox{and} \quad
\begin{pmatrix} 
\frac{\partial g_S}{\partial z_S } &  \frac{\partial g_S}{\partial z_F } \\[0.5ex]
\frac{\partial g_F}{\partial z_S } & \frac{\partial g_F}{\partial z_F }
\end{pmatrix}
\qquad \text{ are regular}
\end{equation}
in a neighborhood of the solution. For later use, we introduce Lipschitz constants with respect to the algebraic variables:
\begin{equation}\label{eq:lipschitz.algebraic}
 || g_S (t, y_S, y_F, z_S, z_F) - g_S (t, y_S, y_F, \hat{z}_S, \hat{z}_F)||
 \,\, \le\,\, L^{g_S}_{S} || z_S -\hat{z}_S|| \,+\,  L^{g_S}_{F} || z_F -\hat{z}_F|| 
\end{equation}
and analogously $L^{g_F}_{S}$, $L^{g_F}_{F}$ and $L^{f_\lambda}_{\rho}$ with $\lambda,\rho\in \{F,\, S\}$. Furthermore, for the Lipschitz constants with respect to the differential variables, we use the symbol $M^j_{\lambda}$ (with 
 $j\in \{f_S,\, f_F
\}$), e.g.,
\begin{equation}\label{eq:lipschitz.differential}
 || f_S (t, y_S, y_F, z_S, z_F) - f_S (t, \hat{y}_S, \hat{y}_F, z_S, z_F)||
 \,\, \le\,\, M^{f_S}_{S} || y_S -\hat{y}_S|| 
        \,+\,  M^{f_S}_{F} || y_F -\hat{y}_F||. 
\end{equation}

To analyze inter-/extrapolation based multirate schemes for these general index-1 DAEs, 
\color{black}
we consider dynamic iteration schemes 
with old, known iterates $y_\lambda^{(i)},\,z_\lambda^{(i)}$ and to be computed, new iterates $y_\lambda^{(i+1)},\, z_\lambda^{(i+1)}$ defined by the following dynamic system
\begin{align*}
\dot y_S^{(i+1)} &= F_S(t,\, y_S^{(i+1)},\, y_F^{(i+1)},\, z_S^{(i+1)},\, z_F^{(i+1)},\, y_S^{(i)},\, y_F^{(i)},\, z_S^{(i)},\, z_F^{(i)}), \\
0 &=  G_S(t,\, y_S^{(i+1)},\, y_F^{(i+1)},\, z_S^{(i+1)},\, z_F^{(i+1)},\, y_S^{(i)},\, y_F^{(i)},\, z_S^{(i)},\, z_F^{(i)}), \\[2mm]
\dot y_F^{(i+1)} &= F_F(t,\, y_S^{(i+1)},\, y_F^{(i+1)},\, z_S^{(i+1)},\, z_F^{(i+1)},\, y_S^{(i)},\, y_F^{(i)},\, z_S^{(i)},\, z_F^{(i)}), \\
0 &=  G_F(t,\, y_S^{(i+1)},\, y_F^{(i+1)},\, z_S^{(i+1)},\, z_F^{(i+1)},\, y_S^{(i)},\, y_F^{(i)},\, z_S^{(i)},\, z_F^{(i)})
\end{align*}
based on splitting functions $F_S,G_S,F_F$ and $G_F$.
To have a simpler notation, we introduce the abbreviations
\[
  x := (y_S,\, y_F,\, z_S,\, z_F).
  \quad
  x_S := (y_S,\, z_S), \quad x_F :=(y_F,\, z_F).
\]
The above splitting functions have to be consistent, this reads,
\begin{alignat*}{2}
F_{\lambda}(t,\, x,\, x) &=f_{\lambda}(t,\, x), 
\qquad & 
G_{\lambda}(t,\, x,\, x) &=g_{\lambda}(t,\, x), 
\quad \text{for } \, \lambda  \in \{F,\,S\}.
\end{alignat*}
For the different multirate approaches, we have the following splitting functions:
\color{black}
\begin{enumerate}[label=\roman*)]
    \item Fully-decoupled approach:
    \begin{alignat*}{2}
F_S(t,\, x^{(i+1)},\, x^{(i)}) & =
    f_S(t,\, x_S^{(i+1)},\, x_F^{(i)}),
\qquad & 
F_F(t,\, x^{(i+1)},\, x^{(i)}) &=
    f_F(t,\, x_S^{(i)},\, x_F^{(i+1)}),
    \\
G_S(t,\, x^{(i+1)},\, x^{(i)}) & =
    g_S(t,\, x_S^{(i+1)},\, x_F^{(i)}),
\qquad & 
G_F(t, x^{(i+1)},\, x^{(i)}) &=
    g_F(t,\, x_S^{(i)},\, x_F^{(i+1)}).
\end{alignat*}
\color{black}
   \item Slowest-first approach:
    \begin{alignat*}{2}
F_S(t,\, x^{(i+1)},\, x^{(i)}) & =
    f_S(t,\, x_S^{(i+1)},\, x_F^{(i)}),
\qquad & 
F_F(t,\, x^{(i+1)},\, x^{(i)}) &=
    f_F(t,\, x_S^{(i+1)},\, x_F^{(i+1)}),
    \\
G_S(t,\, x^{(i+1)},\, x^{(i)}) & =
    g_S(t,\, x_S^{(i+1)},\, x_F^{(i)}),
\qquad & 
G_F(t, x^{(i+1)},\, x^{(i)}) &=
    g_F(t,\, x_S^{(i+1)},\, x_F^{(i+1)}).
\end{alignat*}
\color{black}
 \item Fastest-first approach:
    \begin{alignat*}{2}
F_S(t,\, x^{(i+1)},\, x^{(i)}) & =
    f_S(t,\, x_S^{(i+1)},\, x_F^{(i+1)}),
\qquad & 
F_F(t,\, x^{(i+1)},\, x^{(i)}) &=
    f_F(t,\, x_S^{(i)},\, x_F^{(i+1)}),
    \\
G_S(t,\, x^{(i+1)},\, x^{(i)}) & =
    g_S(t,\, x_S^{(i+1)},\, x_F^{(i+1)}),
\qquad & 
G_F(t, x^{(i+1)},\, x^{(i)}) &=
    g_F(t,\, x_S^{(i)},\, x_F^{(i+1)}).
\end{alignat*}
\color{black}
\end{enumerate}
It has been shown that convergence of a dynamic iteration scheme for DAEs can no longer be guaranteed by choosing a 
window step size $H$ small enough, see e.g. \cite{Arnold_2001,Jackiewicz_1996}. An additional contractivity condition 
has to hold to guarantee convergence. We have to distinguish the following two aspects for contraction: 
\begin{enumerate}[label=\alph*)]
    \item {\em Convergence within one window %
    $[\bar{t},\, \bar{t}+H]$:} 
    In this case, it is sufficient to have~\cite{Jackiewicz_1996}: 
    $$
        \max_{\bar{t} \,\,\le\,\, \tau \,\,\le\,\,\, \bar{t}+H} 
        \left\|
            \left. 
            \begin{pmatrix}
         \frac{\partial G_S}{\partial z_S^{(i+1)}} & \frac{\partial G_S}{\partial z_F^{(i+1)}} \\
           \frac{\partial G_F}{\partial z_S^{(i+1)}} & \frac{\partial G_F}{\partial z_F^{(i+1)}} \\
        \end{pmatrix}^{-1} \cdot
          \begin{pmatrix}
         \frac{\partial G_S}{\partial z_S^{(i)}} & \frac{\partial G_S}{\partial z_F^{(i)}} \\
           \frac{\partial G_F}{\partial z_S^{(i)}} & \frac{\partial G_F}{\partial z_F^{(i)}} \\
        \end{pmatrix} 
        \right|_{}
        \raisebox{-4ex}{%
        $
        \bigl(\tau,\, x(\tau),\, x(\tau)\bigr)$
        }
        \color{black} \right\| %
        \le  \alpha < 1
    $$ 
    using the $L^{\infty}$-norm and evaluation at the analytic solution $x$. 
    The quantity $\alpha \in \real^+$ is referred to as contraction number. 
    For the type of norm employed on the above left-hand side, we use later the following short-hand
    \[
     \left\|\left(\frac{\partial G_\rho\hspace*{2ex}}{\partial x_\lambda^{(i+1)}}\right)^{-1} \frac{\partial G_\lambda}{\partial x_\tau^{(i)}}\right\|
     :=
     \max_{\bar{t} \,\,\le\,\, \tau \,\,\le\,\,\, \bar{t}+H} 
        \left\|
            \left. 
            \left(
         \frac{\partial G_\rho}{\partial x_\lambda^{(i+1)}} 
        \right)^{-1} \cdot
         \frac{\partial G_\lambda}{\partial x_\tau^{(i)}} 
        \right|_{}
        \raisebox{-4ex}{
        $
        \bigl(\tau,\, x(\tau),\, x(\tau)\bigr)$
        }
        \color{black} \right\| %
    \]
    (for $\rho, \lambda, \tau \in \{F,\,S\}, x \in \{y,z\}$).

    \item {\em Stable error propagation from window to window:} %
    Let us assume that $k$ iterations are performed on the current time window. Then 
     a sufficient condition for a stable error propagation from window to window is given by~\cite{Arnold_2001}
    $$
    L_\Phi \alpha^{k} < 1
    $$
    with Lipschitz constant $L_\Phi$ for the extrapolation operator. \color{black}
\end{enumerate}
\begin{remark} i) 
Notice that for the stable error propagation in b) it might be necessary that more than one iteration is performed, although the error reduction (i.e., $\alpha <1$) holds.

ii) If one employs a dynamic iteration with only one iteration (one solve of the DAEs), then a multirate scheme is obtained. These schemes are referred to as \emph{multirate co-simulation}, see \cite{bartel_diss}.
\end{remark}

As we did for the ODE case, interpolation/extrapolation based multirate schemes of convergence order $p$ for coupled index-1 DAEs can now be obtained by replacing the exact solution of the DAE system with splitting functions 
\begin{enumerate}[label=\roman*)]
    \item by a numerical integration of convergence order $p$, 
    
    \item with stopping after the first iteration (i.e.,  $k=1$), plus
    
    \item employing extrapolation/interpolation schemes of order $p-1$  and
    
    \item having satisfied the contractivity condition $L_\Phi \alpha<1$.
\end{enumerate}
For the different coupling strategies, this condition reads
\begin{enumerate}[label=\roman*)]
    \item fully-decoupled approach: 
    \begin{align*}
        L_\Phi
        \max_{\scriptscriptstyle\bar{t} \le\,\, \tau\,\, \le \bar{t}+H} 
        & 
        \left\|
        \setlength\arraycolsep{3pt}
        \begin{pmatrix}
            \frac{\partial G_S}{\partial z_S^{(i+1)}} & 0 \\[1.5ex]
            0 & \frac{\partial G_F}{\partial z_F^{(i+1)}} \\
        \end{pmatrix}^{\!\!\!-1} \cdot
          \begin{pmatrix}
         0 & \frac{\partial G_S}{\partial z_F^{(i)}} \\[1.5ex]
           \frac{\partial G_F}{\partial z_S^{(i)}} & 0 \\
        \end{pmatrix} 
        \right\|  
        \; <  1 
        \\
 & \Leftrightarrow   
 \max\limits_{\scriptscriptstyle\bar{t} \le\,\, \tau\,\, \le \bar{t}+H} 
        \left\|
        \setlength\arraycolsep{3pt}
        \begin{pmatrix}
             \left(\frac{\partial G_S}{\partial z_S^{(i+1)}}\right)^{-1} \frac{\partial G_S}{\partial z_F^{(i)}} & 0 \\[1.5ex]
           0 & \left(\frac{\partial G_F}{\partial z_F^{(i+1)}}\right)^{-1} \frac{\partial G_F}{\partial z_S^{(i)}}  \\
        \end{pmatrix}
         \right\|  
        \; < \frac{1}{L_\Phi}.
 \end{align*}
Sufficient conditions for this are  
 \begin{align*}
\left\| 
 \left(\frac{\partial G_S}{\partial z_S^{(i+1)}}\right)^{\!\!\!-1} \frac{\partial G_S}{\partial z_F^{(i)}} 
\right\| < \frac{1}{L_\Phi} \quad \mbox{and} \quad
\left\| \left(\frac{\partial G_F}{\partial z_F^{(i+1)}}\right)^{\!\!\!-1}  \frac{\partial G_F}{\partial z_S^{(i)}}
\right\| < \frac{1}{L_\Phi}.
\end{align*}    
Introducing the ratios of Lipschitz-constants: \color{black}
$$
\alpha_S:=\frac{L^{g_S}_{F}}{L^{g_S}_{S}}, \qquad \alpha_F:=    \frac{L^{g_F}_{S}}{L^{g_F}_{F}}  
$$
for $g_S$ and $g_F$ (see \eqref{eq:lipschitz.algebraic}), 
the last conditions can be reformulated as:  \color{black}
\begin{equation}\label{eq:lip.fully-decoupled}
    \textstyle 
    \alpha_S < \frac{1}{L_\Phi} \quad \mbox{and} \quad \alpha_F < \frac{1}{L_\Phi}.
\end{equation}

    \item \color{black} slowest-first approach: %
    \begin{align*} 
          \max_{\scriptscriptstyle\bar{t} \,\,\le \,\,\tau \,\,\le\,\, \bar{t}+H} 
          &
          \left\|
            \setlength\arraycolsep{3pt}
            \begin{pmatrix}
                \frac{\partial G_S}{\partial z_S^{(i+1)}} & 0 \\[1.5ex]
                \frac{\partial G_F}{\partial z_S^{(i+1)}} & \frac{\partial G_F}{\partial z_F^{(i+1)}} \\
            \end{pmatrix}^{\!\!\!-1} \cdot \setlength\arraycolsep{3pt}
            \begin{pmatrix}
                0 & \frac{\partial G_S}{\partial z_F^{(i)}} \\[1.5ex]
                0 & 0 \\
            \end{pmatrix} \right\|   
        <  1 
\\
& \Leftrightarrow
\max_{\scriptscriptstyle\bar{t} \,\,\le \,\,\tau \,\,\le\,\, \bar{t}+H} 
    \left\|
        \setlength\arraycolsep{5pt}
        \begin{pmatrix}
        0 & \left(\frac{\partial G_S}{\partial z_S^{(i+1)}}\right)^{\!\!\!-1} \frac{\partial G_S}{\partial z_F^{(i)}} \\
        0 &  \left(\frac{\partial G_F}{\partial z_F^{(i+1)}}\right)^{\!\!\!-1} \frac{\partial G_F}{\partial z_S^{(i+1)}}
\left(\frac{\partial G_S}{\partial z_S^{(i+1)}}\right)^{\!\!\!-1}  \frac{\partial G_S}{\partial z_F^{(i)}} \\
        \end{pmatrix}  \color{black} %
           \right\|   
        <  \frac{1}{L_\Phi}.
\end{align*}
For this, sufficient conditions are  \color{black}
\begin{align*}
& \textstyle 
\left\| 
 \left(\frac{\partial G_S}{\partial z_S^{(i+1)}}\right)^{\!\!\!-1} \frac{\partial G_S}{\partial z_F^{(i)}}  \right\| 
 < \frac{1}{L_\Phi} 
 \;\; \text{and } \;\;
\left\| \left(\frac{\partial G_F}{\partial z_F^{(i+1)}}\right)^{\!\!\!-1} \frac{\partial G_F}{\partial z_S^{(i+1)}}
\left(\frac{\partial G_S}{\partial z_S^{(i+1)}}\right)^{\!\!\!-1}  \frac{\partial G_S}{\partial z_F^{(i)}}
\right\| < \frac{1}{L_\Phi}.
\end{align*}
Formulated with ratios of Lipschitz-constants, we have \color{black}
\begin{equation}\label{eq:lip_slowest-first} 
 \textstyle 
 \alpha_S < \frac{1}{L_\Phi} \quad \mbox{and} \quad \alpha_F \alpha_S < \frac{1}{L_\Phi},
\end{equation}
which is equivalent to
\begin{equation}\label{eq:lip_slowest-first-2} 
 \textstyle 
 \alpha_S < \frac{1}{L_\Phi} \quad \mbox{and} \quad \alpha_F  < 1.
\end{equation}

    \item \color{black}
    fastest-first approach: %
    we obtain analogously to ii) \color{black}
        \begin{align*}
& 
    \left\| \setlength\arraycolsep{5pt}
       \begin{pmatrix}
        \left(\frac{\partial G_S}{\partial z_S^{(i+1)}}\right)^{\!\!\!-1} \frac{\partial G_S}{\partial z_F^{(i+1)}}
\left(\frac{\partial G_F}{\partial z_F^{(i+1)}}\right)^{\!\!\!-1}  \frac{\partial G_F}{\partial z_S^{(i)}}
           & 0 \\
         \left(\frac{\partial G_F}{\partial z_F^{(i+1)}}\right)^{\!\!\!-1} 
      \frac{\partial G_F}{\partial z_S^{(i)}} & 0 
      \end{pmatrix}
    \right\|  \; < \frac{1}{L_\Phi}.
       \\
\end{align*}
For this,  \color{black}
sufficient conditions for this are 
\begin{align*}
& \textstyle 
\left\| 
 \left(\frac{\partial G_F}{\partial z_F^{(i+1)}}\right)^{\!\!\!-1} 
       \frac{\partial G_F}{\partial z_S^{(i)}} 
\right\| < \frac{1}{L_\Phi} \; \text{and} \;
 \left\| \left(\frac{\partial G_S}{\partial z_S^{(i+1)}}\right)^{\!\!\!-1} \frac{\partial G_S}{\partial z_F^{(i+1)}}
\left(\frac{\partial G_F}{\partial z_F^{(i+1)}}\right)^{\!\!\!-1}  \frac{\partial G_F}{\partial z_S^{(i)}}
\right\| < \frac{1}{L_\Phi}.
\end{align*}
In ratios of Lipschitz-constants, this reads \color{black}
\begin{equation}\label{eq:lip_fastest-first}
\textstyle
\alpha_F <  \frac{1}{L_\Phi} \quad \mbox{and} \quad \alpha_S \alpha_F < \frac{1}{L_\Phi},
\end{equation}
which is equivalent to
\begin{equation}\label{eq:lip_fastest-first-2}
 \textstyle 
 \alpha_F <  \frac{1}{L_\Phi} \quad \mbox{and} \quad \alpha_S  < 1.
\end{equation}
\end{enumerate}
\color{black}
In all cases, convergence is given for problems that are coupled weakly enough, i.e., the respective above estimates for $L_\Phi \alpha <1$ hold\color{black}. If not, additional iteration of the multirate scheme will be necessary. This will, in fact, destroy the multirate benefit. 

\begin{remark}
One shall notice \color{black} that the stability criteria are relaxed if the multirate scheme is not fully decoupled: a larger fast ratio $\alpha_F$ is allowed in the case of slowest-first approach, and a larger slow ratio $\alpha_S$ is in the case of fastest-first approach.
\end{remark}

Summing up, we have

\begin{theorem}
Given the split DAE problem \eqref{eq:DAE-split-formulation} with the index-1 conditions
for the overall system and the subsystems~\eqref{eq:cosim-index-one-conditions}\color{black}. The above variants of multirate methods based on dynamic iteration 
are convergent on the macro step level of order $p$ if
\begin{enumerate}[label=\alph*)]
    \item the respective basic integration schemes are of order $p$,
    \item the applied inter-/ extrapolation procedures are of  order $p-1$, and 
    \item the respective step size restriction
\begin{center}
i) fully-decoupled: \eqref{eq:lip.fully-decoupled},
\qquad
ii) slowest-first: \eqref{eq:lip_slowest-first-2},  
\qquad
iii) fastest-first: \eqref{eq:lip_fastest-first-2},
\end{center}
\end{enumerate}
are satisfied. The latter conditions guarantee stability. 
\end{theorem}

\begin{remark}
In the special case of DAE-ODE coupling, $G_S$ and $G_F$ do not 
depend on old iterates of the algebraic variables; hence $\alpha=0$, and convergence can always be guaranteed
for $H$ small enough. 
For the case, where the fast system is an ODE, and implicit Euler approaches are used, explicit conditions for convergence are given in~\cite{HBGS_2019} 
and read in our notation: 
$$
H < \frac{1}{M^{f_S}_{S}\,+\, L^{f_S}_S\, M^{g_S}_{S}}, 
\quad
h <  \frac{1}{M^{f_F}_S \,+\, L^{f_F}_S \, M^{g_S}_F }
$$
We note that this conditions  are quite strong assumptions in the case of stiff equations.  
\end{remark}

\begin{remark}[Schemes]
Compared with the ODE case, the first order extrapolation needs Jacobian information for the $G$-parts. In fact, this is needed for an implicit integration scheme anyways.
\end{remark}

\section{Conclusion and outlook}

The presented work contains a full convergence theory for the quite straightforward approach of inter/extrapolation-based multirate schemes for both the ODE and index-1 DAE case. By linking our theory to the concept of multirate dynamic iteration schemes, we obtained strong stability restrictions for stiff differential equations. As these conditions are sufficient ones, 
one-sided Lipschitz-conditions might yield more realistic results. This will be investigated in future work.


%
%

\end{document}